\documentclass[12pt,english,american]{article}
\usepackage[utf8]{inputenc}
\usepackage{xcolor}
\usepackage{babel}
\usepackage{mathrsfs}
\usepackage{amsmath}
\usepackage{amsthm}
\usepackage{amssymb}
\usepackage{geometry}
\usepackage{setspace}
\usepackage[pdfusetitle,
 bookmarks=true,bookmarksnumbered=false,bookmarksopen=false,
 breaklinks=false,pdfborder={0 0 1},backref=false,colorlinks=true]
 {hyperref}
\hypersetup{
 pdfborderstyle=,linkcolor=blue,urlcolor=cyan,citecolor=red}

\makeatletter
\usepackage{paralist}
\numberwithin{figure}{section}
\numberwithin{table}{section}
\numberwithin{equation}{section}

\date{}
\usepackage[misc]{ifsym}
\everymath{\displaystyle}

\usepackage{color}
\def\mystrut(#1,#2){\vrule height #1pt depth #2pt width 0pt}
\usepackage[normalem]{ulem}
\newcommand\culine{\bgroup\markoverwith
{\textcolor{red}{\rule[-0.5ex]{2pt}{0.5pt}}}\ULon}
\def\cdashuline{\bgroup
\UL@setULdepth
\markoverwith{\textcolor{red}{\kern.13em
\vtop{\kern\ULdepth \hrule width .3em}%
\kern.13em}}\ULon}
\def\cuuline{\bgroup \UL@setULdepth
\markoverwith{\textcolor{red}{\lower\ULdepth\hbox
{\kern-.03em\vbox{\hrule width.2em\kern1.2\p@\hrule}\kern-.03em}}}
\ULon}
\usepackage{enumitem}
\setlist[itemize,1]{label={\fontfamily{cmr}\fontencoding{T1}\selectfont\textbullet}}
\usepackage{amssymb}

\makeatletter
\renewenvironment{proof}[1][\proofname] {\par\pushQED{\qed}\normalfont\topsep6\p@\@plus6\p@\relax\trivlist\item[\hskip\labelsep\bfseries#1\@addpunct{.}]\ignorespaces}{\popQED\endtrivlist\@endpefalse}
\makeatother
\usepackage{xpatch}
\makeatletter
\AtBeginDocument{\xpatchcmd{\@thm}{\thm@headpunct{.}}{\thm@headpunct{.}}{}{}}
\makeatother
\usepackage{mdwlist}

\makeatother

\theoremstyle{plain}
\newtheorem{thm}{\protect\theoremname}[section]
\theoremstyle{definition}
\newtheorem{defn}[thm]{\protect\definitionname}
\theoremstyle{remark}
\newtheorem{rem}[thm]{\protect\remarkname}
\theoremstyle{definition}
\newtheorem{example}[thm]{\protect\examplename}
\theoremstyle{plain}
\newtheorem{prop}[thm]{\protect\propositionname}
\newtheorem{lem}[thm]{\protect\lemmaname}
\newtheorem{lyxalgorithm}[thm]{\protect\algorithmname}
\newtheorem{cor}[thm]{\protect\corollaryname}
\addto\captionsamerican{\renewcommand{\algorithmname}{Algorithm}}
\addto\captionsamerican{\renewcommand{\corollaryname}{Corollary}}
\addto\captionsamerican{\renewcommand{\definitionname}{Definition}}
\addto\captionsamerican{\renewcommand{\examplename}{Example}}
\addto\captionsamerican{\renewcommand{\lemmaname}{Lemma}}
\addto\captionsamerican{\renewcommand{\propositionname}{Proposition}}
\addto\captionsamerican{\renewcommand{\remarkname}{Remark}}
\addto\captionsamerican{\renewcommand{\theoremname}{Theorem}}
\addto\captionsenglish{\renewcommand{\algorithmname}{Algorithm}}
\addto\captionsenglish{\renewcommand{\corollaryname}{Corollary}}
\addto\captionsenglish{\renewcommand{\definitionname}{Definition}}
\addto\captionsenglish{\renewcommand{\examplename}{Example}}
\addto\captionsenglish{\renewcommand{\lemmaname}{Lemma}}
\addto\captionsenglish{\renewcommand{\propositionname}{Proposition}}
\addto\captionsenglish{\renewcommand{\remarkname}{Remark}}
\addto\captionsenglish{\renewcommand{\theoremname}{Theorem}}
\providecommand{\algorithmname}{Algorithm}
\providecommand{\corollaryname}{Corollary}
\providecommand{\definitionname}{Definition}
\providecommand{\examplename}{Example}
\providecommand{\lemmaname}{Lemma}
\providecommand{\propositionname}{Proposition}
\providecommand{\remarkname}{Remark}
\providecommand{\theoremname}{Theorem}

\begin{document}
\title{\textbf{A new theorem of alternatives leading to sufficient conditions
for the superiorization guarantee question of Dynamic String-Averaging
in the inconsistent case}}
\author{\selectlanguage{english}%
Kay Barshad$^{*}$ and Yair Censor\\
Department of Mathematics, University of Haifa, Mt. Carmel, \\
Haifa 3498838, Israel \\
\Letter ~ \href{mailto:kaybarshad@gmail.com}{kaybarshad@gmail.com};
\Letter~ \href{mailto:yair@math.haifa.ac.il}{yair@math.haifa.ac.il}}
\date{\selectlanguage{english}%
November 6, 2025. Revised March 31, 2026. Revised July 13, 2026}
\maketitle
\selectlanguage{american}%
\begin{abstract}
We study the Superiorization Methodology (SM) in the context of the
General Dynamic String-Averaging (GDSA) method in the inconsistent
case (that is, where the input operators don't have a common fixed
point) which primarily aims at achieving convex feasibility while
simultaneously reducing an objective function. In many scientific
and real-world problems modeled as constrained minimization tasks,
striving for the exact constrained optimum can be costly in terms
of time, energy, and resources. Therefore, applying the SM can offer
a practical and efficient alternative. In particular, we present a
new “theorem of alternatives” for the superiorization method which
leads to investigation of theoretical conditions under which the superiorized
version of the GDSA algorithm converges to a “superior” feasible point,
i.e., one with an objective function value that is smaller or equal
to that produced by the unperturbed feasibility-seeking algorithm.
While this question has only been partially addressed in the existing
literature, we present new sufficient conditions that guarantee that
the SM attains such a superior outcome.\bigskip{}

($*)$ Corresponding author: \foreignlanguage{english}{Kay Barshad,
\href{mailto:kaybarshad@gmail.com}{kaybarshad@gmail.com}.}
\end{abstract}

\section{Introduction}\label{sec:Introduction}

The Superiorization Methodology (SM) embeds objective function reduction
steps (“perturbations”) into a feasibility-seeking algorithm (“the
basic algorithm”), producing a ``superiorized version of the basic
algorithm''. These perturbations aim to reduce the objective function
values locally before the next feasibility-seeking iterative step.
While the SM has been widely successful in practice across many large-scale
applications, a full mathematical guarantee that the superiorized
algorithm both converges to a feasible point and achieves a superior
objective function value (that is, a smaller or equal value) compared
to the unperturbed algorithm, in which all other settings are the
same as in the perturbed one, remains unresolved.

Practical applications provide strong empirical evidence that the
SM often achieves global objective function reduction. Many studies
cited in \cite{CensorBibl} demonstrate this in real-world applications.
Partial theoretical results exist, such as strong \foreignlanguage{english}{Fej{\'e}r}
monotonicity (see, for example, \cite[Theorem 5.4]{GDSA_inconsist}
and \cite[Theorem 4.1]{CZ_sup}) and concentration of measure arguments
(\cite{CensorLevy2021}), but counterexamples (see, for instance,
\cite{Artacho_restart}, \cite{neg_condit_2025} and \cite{Torregrosa-2024})
also show that the SM can fail to deliver the guarantee. The recent
study described in \cite{neg_condit_2025} yields a necessary condition
for the SM to fail. The central open question, therefore, is to identify
precise conditions under which the SM is guaranteed to succeed.

The General Dynamic String-Averaging (GDSA) method has been recently
introduced in \cite{GDSA_inconsist}, where its convergence properties
were studied and it was shown there that the sequence of output operators
of this method is \foreignlanguage{english}{Fej{\'e}r }monotone with
respect to a certain feasible set. Although the results in \cite{GDSA_inconsist}
apply to the inconsistent case, where the input operators have no
common fixed points, they are also valid in the consistent case, where
the input operators do have a common fixed point.

In this paper we continue to focus on the inconsistent case, which
(like in \cite{GDSA_inconsist}) requires assumptions on the output
operators. Namely, we present a new “theorem of alternatives” for
the SM. It states that either the superiorized version of the GDSA
algorithm outperforms the feasibility-seeking GDSA algorithm, or,
starting from the same initialization point, the sequence of distances
between their iterates is strictly decreasing. As a consequence of
this theorem, we derive sufficient conditions guaranteeing that the
SM algorithm, based on the GDSA scheme with negative (sub)gradient
descent perturbations, provides a superior outcome compared to the
corresponding unperturbed scheme under the same settings. These sufficient
conditions provide insights into effective strategies for choosing
the perturbation parameters, which we also demonstrate in our work.

It should be emphasized at the outset that the sufficient conditions
presented in Corollaries \ref{subseq_suff_cond} and \ref{inner_prod}
below are primarily theoretical guarantees. In general, they cannot
be verified during the execution of the algorithm. Nevertheless, their
structural meaning suggests heuristic practical indicators (discussed
at the end of Section \ref{sec3}) that may be monitored or encouraged
in implementations. These indicators are heuristic in nature and should
not be interpreted as verifiable criteria equivalent to the sufficient
conditions.

Our results are somewhat restricted in the consistent case, since
they leave room for improvement by imposing admissibility assumptions
on the input operators rather than on the output ones. We leave a
detailed study of the consistent case for future work in a different
framework.

From a computational perspective, SM is particularly valuable in large,
sparse problems where traditional optimization methods become costly.
For instance, the Linear Superiorization (LinSup) approach (see \cite{Cen_Lin_sup})
has been shown to outperform the Simplex algorithm on large-scale
linear programs. The continuously updated bibliography in \cite{CensorBibl}
further attests to SM’s success across practical, large-scale applications.

The original motivation of the SM remains practical: the desire to
reach a feasible point in the intersection of many constraints sets
while achieving low-cost improvements of an objective function---not
necessarily full optimization, but “satisficing”\footnote{Support for the reasoning of the SM may be borrowed from the American
scientist and Nobel-laureate Herbert Simon (see \cite{Simon1956RationalChoice})
who was in favor of ‘satisficing’ rather than ‘maximizing’. Satisficing
is a decision-making strategy that aims for a satisfactory or adequate
result, rather than the optimal solution. This is because aiming for
the optimal solution may necessitate needless expenditure of time,
energy and resources. The term ‘satisfice’ was coined by Simon in
1956, see also: https://en.wikipedia.org/wiki/Satisficing.} reductions. Recent developments, such as randomized perturbations
and step-size restarts, further enhance the SM’s effectiveness without
undermining its resilience to perturbations, such randomization and
step-size restart strategies, are detailed in \cite{Artacho_restart},
\cite{Berin_Kum_Pakk_Sule_2020}, \cite{CensorLevy2021} and \cite{erturk-salim-2023}.
For more information concerning the SM see, for example,\cite{neg_condit_2025},
\cite{CensorBibl}, \cite{asymmetric-2023}, \cite{CZ_sup}, \cite{humphries-2022}
and \cite{Torregrosa-2024}.

Modern studies showcase the practical impact of the SM in various
situations where data of constrained optimization problems has to
be considered. Guenter et al. in \cite{CollinsEtAl2022} compare the
SM with regularization on large, sparse linear systems in image reconstruction.
Fink et al. in \cite{Cavalc_Fink_Stan_2021} use bounded perturbation
resilient mappings with SM-inspired perturbations for nonconvex multicast
beamforming. Pakkaranang et al. in \cite{Berin_Kum_Pakk_Sule_2020}
integrate the SM into a modified proximal gradient algorithm for non-smooth
composite optimization, which they apply to image recovery problems,
achieving strong convergence results. Ma et al. in \cite{AmaranEtAl2021}
extend derivative-free optimization algorithms to large-scale problems
within an SM-related framework. Numerous other applications are summarized
in \cite{CensorBibl}, which include, inter alia, \cite{humphries-2022},
\cite{Janakiev2016}, \cite{Langthaler2014}, \cite{LuoZhou2014}
and \cite{Oliveira2023}, to mention but a few.

The rest of the paper is organized as follows. Section \ref{sec2}
provides background on the algorithms and operators under consideration,
which is required for establishing our results. In Section \ref{sec3}
our new theorem of alternatives is formulated and proved which enables
comparing the behavior of the superiorized algorithm with the unperturbed
feasibility-seeking one. This is followed by our sufficient conditions
for the guarantee question of\textbf{ }the superiorization methodology
for Dynamic String-Averaging in the inconsistent case, discussed above.

\section{Preliminaries}\label{sec2}

Throughout this paper $\mathbb{N}$ denotes the set of natural numbers
(starting from $0$), and for any two integers $m$ and $n$, with
$m\le n$, we denote by $\left\{ m,m+1,\dots,n\right\} $ the set
of all integers between $m$ and $n$. The real line is denoted by
$\mathbb{R}$. For a set $A$, we denote by $\left|A\right|$ the
cardinality of $A$. For a real Hilbert space $\mathcal{H}$, we use
the following notations:
\begin{itemize}
\item $\langle\cdot,\cdot\rangle$ denotes the inner product on $\mathcal{H}.$
\item $\Vert\cdot\Vert$ denotes the norm on $\mathcal{H}$ induced by $\langle\cdot,\cdot\rangle.$
\item $Id$ denotes the identity operator on $\mathcal{H}$.
\item $\mathrm{Fix}T$ denotes the, possibly empty, set $\mathrm{Fix}T:=\{x\in\mathcal{H}\mid T(x)=x\}$
of fixed points of an operator $T:\mathcal{H}\rightarrow\mathcal{H}$.
\item For a nonempty and convex subset $C$ of $\mathcal{H}$, we denote
by $P_{C}$ the (unique) metric projection onto $C$, the existence
of which is guaranteed if $C$ is, in addition, closed.
\item For a convex function $\phi:\mathcal{H}\rightarrow\mathbb{R}$ and
a point $x\in\mathcal{H}$, we denote by $\partial\phi\left(x\right)$
the subdifferential of $\phi$ at $x$, that is,
\begin{equation}
\partial\phi\left(x\right):=\left\{ g\in\mathcal{H}\,|\,\left\langle g,y-x\right\rangle \le f\left(y\right)-f\left(x\right)\,\,\mathrm{for}\,\,\mathrm{all}\,\,y\in\mathcal{H}\right\} .\label{eq:-24}
\end{equation}
\item $B\left(x,r\right)$ denotes the open ball centered at $x\in\mathcal{H}$
of radius $r>0$.
\item For a function $f:\mathcal{H}\rightarrow\mathbb{R}$ and a subset
$A$ of $\mathcal{H}$, we denote by $\underset{x\in A}{\mathrm{Argmin}}f\left(x\right)$
and $\underset{x\in A}{\mathrm{Argmax}}f\left(x\right)$, respectively,
the set of minimizers of $f$ and the set of maximizers of $f$ on
the set $A$.
\end{itemize}
We recall the following types of algorithmic operators. For more information
on such operators, see, for example, \cite{C_book}.
\begin{defn}
Let $T:\mathcal{H}\rightarrow\mathscr{\mathcal{H}}$ be an operator
and let $\lambda\in\left[0,2\right]$. The operator $T_{\lambda}:\mathcal{H}\rightarrow\mathscr{\mathcal{H}}$
defined by $T_{\lambda}:=\left(1-\lambda\right)Id+\lambda T$ is called
a $\lambda$-$relaxation$ of the operator $T$. The operator $T_{2}$
is called a \textit{reflection} of the operator $T$.
\end{defn}

The operator $T_{\lambda}$ in this definition is an averaged mapping
in the sense of Section 2 in \cite{BaillonBruckReich1977}.
\begin{defn}
An operator $T:\mathcal{H}\rightarrow\mathcal{H}$ is said to be \textit{nonexpansive}
(NE) if 
\[
\left(\forall x,y\in\mathcal{H}\right)\,\,\left\Vert T\left(x\right)-T\left(y\right)\right\Vert \leq\left\Vert x-y\right\Vert .
\]
For $\lambda\in\left[0,2\right]$, an operator $T:\mathcal{H}\rightarrow\mathcal{H}$
is said to be $\lambda$-\textit{relaxed nonexpansive} if $T$ is
a $\lambda$-relaxation of a nonexpansive operator $U$, that is,
$T=U_{\lambda}$.
\end{defn}

\begin{rem}
\label{conv_com_comp}Clearly, convex combinations and compositions
of nonexpansive operators are also nonexpansive. Moreover, for each
$\lambda\in\left[0,1\right]$, the $\lambda$-relaxation of a nonexpansive
operator is also nonexpansive.
\end{rem}

\begin{defn}
We say that an operator $T:\mathscr{\mathcal{H}}\rightarrow\mathcal{H}$
is \textit{firmly nonexpansive} if
\[
\left(\forall x,y\in\mathcal{H}\right)\left\langle T\left(x\right)-T\left(y\right),x-y\right\rangle \ge\left\Vert T\left(x\right)-T\left(y\right)\right\Vert ^{2}.
\]
\end{defn}

\selectlanguage{english}%
\begin{thm}
[{\cite[Theorem 2.2.10]{C_book}}]\foreignlanguage{american}{\label{FNE-equiv.cond}Let
$T:\mathcal{H}\rightarrow\mathcal{H}$ be an operator. The following
conditions are equivalent:}
\selectlanguage{american}%
\begin{enumerate}
\item $T$ is firmly nonexpansive.
\item $T_{\lambda}$ is nonexpansive for each $\lambda\in\left[0,2\right]$.
\item There exists a nonexpansive operator $N:\mathcal{H\rightarrow\mathcal{H}}$
such that $T=2^{-1}\left(Id+N\right)$.
\end{enumerate}
\end{thm}

\selectlanguage{american}%
Theorem \ref{FNE-equiv.cond} is essentially Proposition 11.2 on page
42 of the book by Goebel and Reich, \cite{GoebelReich1984}.
\begin{example}
\label{metric projection}Given a nonempty, closed an convex subset
$C$ of $\mathcal{H}$, the metric projection $P_{C}$ onto $C$ is
firmly nonexpansive (see, e.g., Theorem 2.2.21 in \cite{C_book}).
Moreover $\mathrm{Fix}P_{C}=C$.
\end{example}

\begin{defn}
For $\lambda\in\left[0,2\right]$, an operator $T:\mathcal{H}\rightarrow\mathcal{H}$
is called $\lambda$-\textit{relaxed firmly nonexpansive} if $T$
is a $\lambda$-relaxation of a firmly nonexpansive operator $U$,
that is, $T=U_{\lambda}=\left(1-\lambda\right)Id+\lambda U$.
\end{defn}

\begin{defn}
We say that an operator $T:\mathcal{H}\rightarrow\mathcal{H}$ is
\textit{quasi-nonexpansive} if 
\[
\left\Vert T\left(x\right)-z\right\Vert \le\left\Vert x-z\right\Vert 
\]
for each $x\in\mathcal{H}$ and $z\in\mathrm{Fix}T$.
\end{defn}

The notions of bounded regularity and approximate shrinking are needed
for establishing the strong convergence of the methods which we discuss
in this paper. The bounded regularity of a finite family of sets was
studied in \cite[Section 5]{BB96} and \cite{BB93}. This property
was extended to an infinite family of sets in \cite{GDSA_inconsist}.
We recall it briefly below.
\begin{defn}
\label{bound_reg}For a nonempty index set $I$, the family $\left\{ C_{i}\right\} _{i\in I}$
of nonempty, closed and convex subsets of $\mathcal{H}$ with nonempty
intersection $C$ is \textit{boundedly regular }if for any bounded
sequence $\left\{ x^{k}\right\} ^{\infty}_{k=0}$ in $\mathcal{H}$,
the following implication holds:

\textit{
\[
\lim_{k\rightarrow\infty}d\left(x^{k},C_{i}\right)=0\,\,\mathrm{for}\,\,\mathrm{each\,\,}i\in I\,\,\Longrightarrow\,\,\lim_{k\rightarrow\infty}d\left(x^{k},C\right)=0.
\]
}
\end{defn}

The next proposition provides sufficient conditions for bounded regularity.
We recall that a topological space $X$ is locally compact if each
$x\in X$ has a compact neighborhood with respect to the topology
inherited from $X$.
\selectlanguage{english}%
\begin{prop}
[{\cite[Proposition 3.2]{GDSA_inconsist}}]\foreignlanguage{american}{\label{bounded_reg}Let
$\left\{ C_{i}\right\} _{i\in I}$ be a family of nonempty, closed
and convex subsets of $\mathcal{H}$ with a nonempty intersection
$C$. Then the following assertions hold:}
\selectlanguage{american}%
\begin{enumerate}
\item If there is $i_{0}\in I$ for which the set $C_{i_{0}}$ is a locally
compact topological space (with respect to the norm topology inherited
from $\mathcal{H}$), then the family $\left\{ C_{i}\right\} _{i\in I}$
is boundedly regular.
\item If $\mathcal{H}$ is of finite dimension, then the family $\left\{ C_{i}\right\} _{i\in I}$
is boundedly regular.
\end{enumerate}
\end{prop}

\selectlanguage{american}%
We also recall the following notion of approximate shrinking which
was extensively studied in \cite{CiegZ_app_s}. This notion is also
known in the literature as regularity (see, for instance, \cite[Definition 7.1]{BNP15}
or \cite[Definition 3.1]{CRZ18}). More information on boundedly regular
operators and families of mappings, and their applications, can be
found, for example, in Section 3 of \cite{ReichSalinas2017}.
\begin{defn}
A quasi-nonexpansive operator $T:\mathcal{H}\rightarrow\mathcal{H}$
\textit{is approximately shrinking} if for each bounded sequence $\left\{ x^{k}\right\} ^{\infty}_{k=0}$
in $\mathcal{H}$, the following implication holds:
\[
\lim_{k\rightarrow\infty}\left\Vert T\left(x^{k}\right)-x^{k}\right\Vert =0\,\,\Longrightarrow\,\,\lim_{k\rightarrow\infty}d\left(x^{k},\mathrm{Fix}T\right)=0.
\]
\end{defn}

\begin{example}
\label{ex_as}Given a nonempty, closed an convex subset $C$ of $\mathcal{H}$,
the metric projection $P_{C}$ onto $C$ is approximately shrinking
(see Example 3.5 in \cite{CiegZ_app_s}).
\end{example}

\selectlanguage{english}%
\begin{thm}
[{\cite[Theorem 16.7(ii)]{BC_book}}]\foreignlanguage{american}{\label{subdiff_ne+b}Let
$\phi:\mathcal{H}\rightarrow\mathbb{R}$ be convex and continuous
function at the point $x\in\mathcal{H}$. Then the subgradient set
$\partial\phi\left(x\right)$ is nonempty.}
\end{thm}

\selectlanguage{american}%
\selectlanguage{english}%
\begin{defn}
[{\cite[Definition 2.24]{GDSA_inconsist}}]\foreignlanguage{american}{Let
$\Gamma\subseteq\mathcal{H}$ be a given nonempty subset of $\mathcal{H}$
and $\left\{ T_{k}\right\} ^{\infty}_{k=0}$ be a sequence of operators,
$T_{k}:\mathcal{H}\rightarrow\mathcal{H}$ for each $k\in\mathbb{N}$.
The algorithm $x^{k+1}:=T_{k}(x^{k}),$ for all $k\in\mathbb{N},$
is said to be\textit{ bounded perturbations resilient} with respect
to\textbf{ $\Gamma$}\emph{ }if the following is true: If a sequence
$\{x^{k}\}^{\infty}_{k=0},$ generated by the algorithm, converges
in the norm of $\mathcal{H}$ to a point in $\Gamma$ for all $x^{0}\in\mathcal{H}$,
then any sequence $\{y^{k}\}^{\infty}_{k=0}$ in $\mathcal{H}$ that
is generated by the algorithm $y^{k+1}:=T_{k}(y^{k}+\beta_{k}v^{k}),$
for all $k\in\mathbb{N},$ also converges in the norm of $\mathcal{H}$
to a point in $\Gamma$ for all $y^{0}\in\mathcal{H},$ provided that
$\{\beta_{k}v^{k}\}^{\infty}_{k=0}$ are bounded perturbations, meaning
that $\left\{ \beta_{k}\right\} ^{\infty}_{k=0}$ is a sequence of
positive real numbers such that $\sum^{\infty}_{k=0}\beta_{k}<\infty$
and that the vector sequence $\{v^{k}\}^{\infty}_{k=0}$ is a bounded
sequence in $\mathcal{H}$.}
\end{defn}

\selectlanguage{american}%
The notion of perturbation resilience relies on ideas presented in \cite{BuRZ1}, \cite{ButnariuReichZaslavski2007}
and \cite{Combettes2001}, wherein it was shown that if iterates of
a nonexpansive operator converge for any initial point, then its inexact
iterates with summable errors also converge. For a more detailed history
consult, e.g., \cite{BarshadCensor2026}.

Throughout the rest of the paper we refer, among other things, to
the following setting defined below. We recall that for an arbitrary
sequence of sets $\left\{ A_{k}\right\} ^{\infty}_{k=0}$, $\limsup_{k\rightarrow\infty}A_{k}:=\cap^{\infty}_{n=0}\cup^{\infty}_{k=n}A_{k}$.

Let $m$ be a positive integer. We consider a finite family $\left\{ U_{i}\right\} ^{m}_{i=1}$
of $\alpha_{i}$-relaxed firmly nonexpansive operators, where $U_{i}:\mathcal{H\rightarrow\mathcal{H}}$
and $\alpha_{i}\in\left(0,2\right]$ for each $i=1,2,\dots,m$. Let
$\left\{ q_{k}\right\} ^{\infty}_{k=0}$ be a bounded sequence of
positive integers. By Theorem \ref{FNE-equiv.cond}, the operator
$U_{i}$ is nonexpansive for each $i=1,2,\dots,m$. Let $\mathcal{M}:=\max_{k\in\mathbb{N}}q_{k}$
and define
\[
\rho_{\left\{ U_{i}\right\} ^{m}_{i=1}}:=\min\left\{ \mathcal{M}^{-1}\min_{i\in\left\{ 1,2,\dots,m\right\} }\left(2-\alpha_{i}\right)\alpha^{-1}_{i},1\right\} \le1,
\]
where the inequality follows immediately from the definition, since
$\rho_{\left\{ U_{i}\right\} ^{m}_{i=1}}$ is defined as the minimum
of two numbers, one of which equals 1. Let $\left\{ \varOmega_{k}\right\} ^{\infty}_{k=0}$
be a family of nonempty sets such that $\varOmega_{k}\subset\left\{ 1,2,\dots,m\right\} {}^{\left\{ 1,2,\dots,q_{k}\right\} }.$
That is, $\varOmega_{k}$ is a finite subset of the set of functions
from $\left\{ 1,2,\dots,q_{k}\right\} $ to $\left\{ 1,2,\dots,m\right\} $
for each $k\in\mathbb{N}$. Since the sequence $\left\{ q_{k}\right\} ^{\infty}_{k=0}$
is bounded, the number of different elements in the family $\left\{ \varOmega_{k}\right\} ^{\infty}_{k=0}$
is finite. For each $k\in\mathbb{N}$ and each $t\in\varOmega_{k}$,
set $V_{k}\left[t\right]:=U_{t\left(q_{k}\right)}\cdots U_{t\left(2\right)}U_{t\left(1\right)}$
and let $\omega_{k}:\varOmega_{k}\rightarrow\left(0,1\right]$ be
a function such that $\sum_{t\in\varOmega_{k}}\omega_{k}\left(t\right)=1$.

For each $k\in\mathbb{N}$, define $T_{\left(\varOmega_{k},\omega_{k}\right)}:=\sum_{t\in\varOmega_{k}}\omega_{k}\left(t\right)V_{k}\left[t\right]$.
Define the set $I:=\limsup_{k\rightarrow\infty}\left\{ T_{\left(\varOmega_{k},\omega_{k}\right)}\right\} $,
that is, $I=\limsup_{k\rightarrow\infty}A_{k}$, where the sequence
of sets $\left\{ A_{k}\right\} ^{\infty}_{k=0}$ is defined by singletons,
$A_{k}:=\left\{ T_{\left(\varOmega_{k},\omega_{k}\right)}\right\} $.
Define a family of sets $\left\{ C_{T}\right\} _{T\in I}$ by $C_{T}:=\mathrm{Fix}T$
for each $T\in I$. Set $F:=\cap^{\infty}_{k=0}\mathrm{Fix}T_{\left(\varOmega_{k},\omega_{k}\right)}$
and $C:=\cap_{T\in I}C_{T}$ (in case $I$ is empty $\cap_{T\in I}C_{T}=\mathcal{H}$
by definition). At this stage, no assumption is made on the non-emptiness
of the sets $I,$ $F$ and $C$. Such assumptions are made in the
sequel.
\begin{rem}
The operators $T_{\left(\varOmega_{k},\omega_{k}\right)}$, defined
above, are ``string-averaging operators'' as first introduced in
\cite{CEH2001} and further studied in various forms and settings,
see, for instance, Example 5.21 in \cite{BC_book}, \cite{KPH} and
\cite{Nikazad2016}, to name but a few. In those and other papers,
the index vector $t$ is called ``a string'', the composite operator
$V_{k}\left[t\right]$ is called ``a string operator'' and $\omega_{k}$
are called ``weight functions''.

We introduce the following definition of \textit{$\limsup$}-admissibility
of sequences of operators. The almost cyclic control \cite[Definition 5.6.10]{C_book}
is a particular case of the admissible control \cite[Definition 3.2]{PLC}
used in this definition, while this admissible control itself is a
refinement of the repetitive control \cite[Definition 5.6.11]{C_book}.
In particular, limsup-admissibility guarantees repetitive activation
of all asymptotically relevant operators, without imposing a uniform
window condition.
\end{rem}

\begin{defn}
[{\cite[Definition 4.2]{GDSA_inconsist}}]\label{def:admis}We say
that a sequence $\left\{ T_{k}\right\} ^{\infty}_{k=0}$of operators,
$T_{k}:\mathcal{H}\rightarrow\mathcal{H}$ for each $k\in\mathbb{N}$,
is \textit{$\limsup$-admissible}, if $\left\{ T_{k}\right\} ^{\infty}_{k=0}\subset\limsup_{k\rightarrow\infty}\left\{ T_{k}\right\} $
and for each $T\in\limsup_{k\rightarrow\infty}\left\{ T_{k}\right\} $,
there is an integer $M_{T}>0$ such that $T\in\cup^{k+M_{T}-1}_{n=k}\left\{ T_{\left(\varOmega_{n},w_{n}\right)}\right\} $
for all $k\in\mathbb{N}$.
\end{defn}

\begin{rem}
\label{rem:admiss}Clearly, for each $k_{0}\in\mathbb{N}$, $\limsup_{k\rightarrow\infty}\left\{ T_{k}\right\} =\limsup_{k\rightarrow\infty}\left\{ T_{k_{0}+k}\right\} $
and if a sequence $\left\{ T_{k}\right\} ^{\infty}_{k=0}$ of operators
is $\limsup$-admissible, then the sequence $\left\{ T_{k_{0}+k}\right\} ^{\infty}_{k=0}$
is $\limsup$-admissible.
\end{rem}

\begin{rem}
Observe that if, in the above setting, we require the sequence $\left\{ \omega_{k}\right\} ^{\infty}_{k=0}$
to attain a finite number of values, then the sequence $\left\{ T_{\left(\varOmega_{k},\omega_{k}\right)}\right\} ^{\infty}_{k=0}$
will also contain a finite number of operators. Thus, in order to
ensure the existence of $k_{0}\in\mathbb{N}$ such that the sequence
$\left\{ T_{\left(\varOmega_{k_{0}+k},\omega_{k_{0}+k}\right)}\right\} ^{\infty}_{k=0}$
is $\limsup$-admissible, we only need to require the existence of
an integer $M_{T}>0$, for each $T\in I$, such that $T\in\cup^{n+M_{T}-1}_{n=k}\left\{ T_{\left(\varOmega_{n},w_{n}\right)}\right\} $
for all $k\in\mathbb{N}$.
\end{rem}

\selectlanguage{english}%
\begin{lem}
[{\cite[Lemma 4.7]{GDSA_inconsist}}]\foreignlanguage{american}{\label{lem:relax_output}Let
$\left\{ \lambda_{k}\right\} ^{\infty}_{k=0}$ be a sequence of real
numbers such that $\lambda_{k}\in\left[\varepsilon,1+\rho_{\left\{ U_{i}\right\} ^{m}_{i=1}}-\varepsilon\right]$
for each $k\in\mathbb{N}$, where $\varepsilon>0$. Then the operator
$T_{\left(\varOmega_{k},\omega_{k}\right)\lambda_{k}}$ is a nonexpansive
operator for each $k\in\mathbb{N}$.}
\end{lem}

\selectlanguage{american}%
We consider the properties of the following algorithm which is the
algorithmic framework for our unperturbed GDSA method.
\begin{lyxalgorithm}
[{\cite[Algorithm 4.5]{GDSA_inconsist}}]\label{alg:GDSA}

Given $\varepsilon\in\left(0,1\right]$, $x^{0}\in\mathcal{H}$ and
a sequence $\left\{ T_{\left(\varOmega_{k},\omega_{k}\right)}\right\} ^{\infty}_{k=0}$
of operators, the algorithm is defined by the recurrence
\[
x^{k+1}:=x^{k}+\lambda_{k}\left(T_{\left(\varOmega_{k},\omega_{k}\right)}\left(x^{k}\right)-x^{k}\right),
\]
where $\lambda_{k}\in\left[\varepsilon,1+\rho_{\left\{ U_{i}\right\} ^{m}_{i=1}}-\varepsilon\right]$
for each $k\in\mathbb{N}$.
\end{lyxalgorithm}

Let $\phi:\mathcal{H}\rightarrow\mathbb{R}$ be a convex and continuous
real valued objective function. We investigate the following algorithm
which provides a general framework for superiorization methods with
negative (sub)gradient perturbation. It is the superiorized version
of Algorithm \ref{alg:GDSA}.
\begin{lyxalgorithm}
[{\cite[Algorithm 5.1]{GDSA_inconsist}}]\label{supGDSA}Given $y^{0}\in\mathcal{H}$,
a sequence $\left\{ N_{k}\right\} ^{\infty}_{k=0}$ of positive integers,
a sequence $\left\{ \lambda_{k}\right\} ^{\infty}_{k=0}$ of positive
numbers and a family of positive real sequences $\left\{ \left\{ \beta_{k,n}\right\} ^{N_{k}}_{n=1}\right\} ^{\infty}_{k=0}$
such that $\sum^{\infty}_{k=0}\sum^{N_{k}}_{n=1}\beta_{k,n}<\infty$,
the algorithm is defined by the recurrences
\[
y^{k+1}:=y^{k}+\sum^{N_{k}}_{n=1}\beta_{k,n}v^{k,n}+\lambda_{k}\left(T_{\left(\varOmega_{k},\omega_{k}\right)}\left(y^{k}+\sum^{N_{k}}_{n=1}\beta_{k,n}v^{k,n}\right)-y^{k}-\sum^{N_{k}}_{n=1}\beta_{k,n}v^{k,n}\right)
\]
\[
\mathrm{wherein}
\]
\begin{equation}
v^{k,n+1}:=\begin{cases}
-\left\Vert s^{k,n}\right\Vert ^{-1}s^{k,n}, & \mathrm{if}\,\,0\not\in\partial\phi\left(y^{k}+\sum^{n}_{i=1}\beta_{k,i}v^{k,i}\right),\\
0, & \mathrm{if}\,\,0\in\partial\phi\left(y^{k}+\sum^{n}_{i=1}\beta_{k,i}v^{k,i}\right),
\end{cases}\label{eq:-6}
\end{equation}
for each $k\in\mathbb{N}$ and each $n=0,1,\dots,N_{k}-1$, where
$s^{k,n}$ is a selection of the subgradient $\partial\phi\left(y^{k}+\sum^{n}_{i=1}\beta_{k,i}v^{k,i}\right)$
(which exists by Theorem \ref{subdiff_ne+b}) for each $k\in\mathbb{N}$
and each $n=0,1,\dots,N_{k}-1$ (recalling that, by definition, $\sum^{0}_{i=1}\beta_{k,i}v^{k,i}:=0$).
\end{lyxalgorithm}

\selectlanguage{english}%
\begin{thm}
[{\cite[Theorem 4.8]{GDSA_inconsist}}]\foreignlanguage{american}{\label{thm_conv}Assume
that $C\not=\emptyset$ and that the sequence $\left\{ T_{\left(\varOmega_{k},\omega_{k}\right)}\right\} ^{\infty}_{k=0}$
is $\limsup$-admissible. Let $\left\{ \lambda_{k}\right\} ^{\infty}_{k=0}$
be a sequence of real numbers such that $\lambda_{k}\in\left[\varepsilon,1+\rho_{\left\{ U_{i}\right\} ^{m}_{i=1}}-\varepsilon\right]$
for each $k\in\mathbb{N}$, where $\varepsilon>0$, and let $x^{0}=y^{0}\in\mathcal{H}$.
Suppose that $\left\{ x^{k}\right\} ^{\infty}_{k=0}$ and $\left\{ y^{k}\right\} ^{\infty}_{k=0}$
are the sequences generated, respectively, by Algorithm \ref{alg:GDSA}
and Algorithm \ref{supGDSA} with respect to the sequence $\left\{ \lambda_{k}\right\} ^{\infty}_{k=0}$.
Then the following assertions hold:}
\selectlanguage{american}%
\begin{enumerate}
\item The sequence $\left\{ x^{k}\right\} ^{\infty}_{k=0}$ converges weakly
to a point $x\in C$.
\item If each $T\in I$ is approximately shrinking and the family $\left\{ C_{T}\right\} _{T\in I}$
is boundedly regular, then the convergence is (i) is strong.
\item The sequence$\left\{ x^{k}\right\} ^{\infty}_{k=0}$ is weakly (strongly,
if the convergence in (i) is strong) bounded perturbations resilient
with respect to $C$.
\end{enumerate}
\end{thm}

\selectlanguage{american}%
\selectlanguage{english}%
\begin{rem}
\label{4.10}\foreignlanguage{american}{~}
\selectlanguage{american}%
\begin{enumerate}
\item In particular, if $U_{i}$ is a $2\left(\mathcal{M}+1\right)^{-1}$-relaxed
firmly nonexpansive, then $\rho_{\left\{ U_{i}\right\} ^{m}_{i=1}}=1$
and $\lambda_{k}\in\left[\varepsilon,2-\varepsilon\right]$ for each
$k\in\mathbb{N}$. If $U_{i}$ is a nonexpansive for each $i=1,2,\dots,m$,
then $\rho_{\left\{ U_{i}\right\} ^{m}_{i=1}}=0$ and $\lambda_{k}\in\left[\varepsilon,1-\varepsilon\right]$
for each $k\in\mathbb{N}$.
\item If the space $\mathcal{H}$ is of a finite dimension, then the convergence
in \textit{(i)} is strong.
\item Theorem \ref{thm_conv}\textit{(ii)} presents a somewhat analogous
result to Theorem 4.1 in \cite{MSA}, where the consistent case was
considered, that is input operators were assumed to have a common
fixed point. Theorem \ref{thm_conv}\textit{(ii)} applies here to
the inconsistent case, where the assumption of a common fixed point
of the input operators is replaced by the non-emptiness of the intersection
of fixed point sets of the output operators.
\end{enumerate}
\end{rem}

\selectlanguage{american}%
The following corollary is a direct consequence of Theorem \ref{thm_conv}.
\begin{cor}
\label{basic cor}Let $y^{0}\in\mathcal{H}$. Assume that the sequence
$\left\{ T_{\left(\varOmega_{k},\omega_{k}\right)}\right\} ^{\infty}_{k=0}$
is $\limsup$-admissible and $C\not=\emptyset$. Let $\left\{ \lambda_{k}\right\} ^{\infty}_{k=0}$
be a sequence of real numbers such that $\lambda_{k}\in\left[\varepsilon,1+\rho_{\left\{ U_{i}\right\} ^{m}_{i=1}}-\varepsilon\right]$
for each $k\in\mathbb{N}$, where $\varepsilon>0$. Then the sequence
$\left\{ y^{k}\right\} ^{\infty}_{k=0}$ generated by Algorithm \ref{supGDSA}
converges weakly to a point $y\in C$. If, in addition, each $T\in I$
is approximately shrinking and the family $\left\{ C_{T}\right\} _{T\in I}$
is boundedly regular, then $\left\{ y^{k}\right\} ^{\infty}_{k=0}$
converges strongly to a point $y\in C$.
\end{cor}

In practice, the conditions of bounded regularity and limsup-admissibility
are often satisfied in common settings. For instance, bounded regularity
holds automatically in finite-dimensional Hilbert spaces. The limsup-admissibility
is typically satisfied for sequences of operators generated by periodic
or cyclic control schemes, or more generally when the selection of
operators avoids pathological accumulation of “inactive” operators
over iterations. These observations indicate that the assumptions
used in our convergence analysis, while abstract, are applicable in
a wide range of practical settings.

\section{Sufficient conditions for the guarantee question}\label{sec3}

The guarantee question in our case is to ascertain under which conditions
Algorithm \ref{supGDSA} performs ``better'' than Algorithm \ref{alg:GDSA},
in the sense that the value of the objection function $\phi$ at the
limit point of Algorithm \ref{supGDSA} is smaller than or equal to
the value of $\phi$ at the limit point of Algorithm \ref{alg:GDSA}.
Below we investigate conditions under which we can establish a mathematical
guarantee for this question. We begin by proving the following auxiliary
result, which is a modification of Lemma 5.3 in \cite{GDSA_inconsist}.
\begin{lem}
\label{aux_lem} For an arbitrary nonempty subset $C$ of $\mathcal{H}$,
a convex and continuous function $\phi:\mathcal{H}\rightarrow\mathbb{R}$
and $x,y\in\mathcal{H}$ such that $\phi\left(y\right)>\phi\left(x\right)$,
there exist real numbers $r_{1}>0$, $r_{2}>0$ and a neighborhood
$V$ of $x$ so that for each $\overline{y}\in B\left(y,r_{1}\right)$
and $v\in\partial\phi\left(\overline{y}\right)$, the following assertions
are satisfied:
\begin{enumerate}
\item $0\not\in\partial\phi\left(\overline{y}\right)$ and for each $\overline{z}\in B\left(x,r_{2}\right)$,
\begin{equation}
\left\langle \left\Vert v\right\Vert ^{-1}v,\overline{z}-\overline{y}\right\rangle <0.\label{eq:-7-1}
\end{equation}
\item For each $z\in V$, 
\begin{equation}
\left\langle \left\Vert v\right\Vert ^{-1}v,z-\overline{y}\right\rangle <-2^{-1}r_{2}.\label{eq:-16}
\end{equation}
\item Let $p$ be a non-negative integer. Assume that $\left\{ \alpha_{n}\right\} ^{p}_{n=1}$
is a finite sequence of positive real numbers such that $\sum^{p}_{n=1}a_{n}<2^{-1}r_{1}$
and $\left\{ v^{n}\right\} ^{p}_{n=1}\subset\mathcal{H}\backslash\left\{ 0\right\} $
is a sequence such that $v^{n}\in\partial\phi\left(\overline{y}-\sum^{n-1}_{i=1}\alpha_{i}\left\Vert v^{i}\right\Vert ^{-1}v^{i}\right)$
for each $n=1,2,\dots,p$. If, in addition, $\overline{y}\in B\left(y,2^{-1}r_{1}\right)$,
then
\begin{equation}
\left\Vert \overline{y}-\sum^{p}_{n=1}\alpha_{n}\left\Vert v^{n}\right\Vert ^{-1}v^{n}-z\right\Vert ^{2}\le\left\Vert \overline{y}-z\right\Vert ^{2}-\sum^{p}_{n=1}\left(r_{2}-\alpha_{n}\right)\alpha_{n}\label{eq:-25}
\end{equation}
for each $z\in V$ (by definition $\sum^{0}_{n=1}\alpha_{n}\left\Vert v^{n}\right\Vert ^{-1}v^{n}:=\sum^{0}_{n=1}\left(r_{2}-\alpha_{n}\right)\alpha_{n}:=0$).
\end{enumerate}
\end{lem}

\begin{proof}
Since $\phi\left(y\right)-\phi\left(x\right)>0$. By the continuity
of $\phi$, there exist $r_{1}>0$ and $r_{2}>0$ such that
\begin{equation}
\phi\left(\overline{y}\right)-\phi\left(\overline{z}\right)>0,\label{eq:-15}
\end{equation}
for each $\overline{y}\in B\left(y,r_{1}\right)$ and $\overline{z}\in B\left(x,r_{2}\right)$.
Let $\overline{y}\in B\left(y,r_{1}\right)$ and $v\in\partial\phi\left(\overline{y}\right)$.
Set $V:=B\left(x,2^{-1}r_{2}\right)$.

\textit{(i)} In view of (\ref{eq:-15}) and (\ref{eq:-24}), we have
for each $\overline{z}\in B\left(x,r_{2}\right)$,
\[
\left\langle v,\overline{z}-\overline{y}\right\rangle <0.
\]
It follows that $v\not=0$ and (\ref{eq:-7-1}) holds. Since $v$
is an arbitrary element of $\partial\phi\left(\overline{y}\right)$,
it follows that $0\not\in\partial\phi\left(\overline{y}\right)$.

\textit{(ii)} Let $z\in V$. Define $\overline{z}:=z+2^{-1}r_{2}\left\Vert v\right\Vert ^{-1}v$.
Then by the triangle inequality, $\overline{z}\in B\left(x,r_{2}\right)$
and by (\ref{eq:-7-1}),
\[
\left\langle \left\Vert v\right\Vert ^{-1}v,z+2^{-1}r_{2}\left\Vert v\right\Vert ^{-1}v-\overline{y}\right\rangle =\left\langle \left\Vert v\right\Vert ^{-1}v,\overline{z}-\overline{y}\right\rangle <0
\]
and (\ref{eq:-16}) follows.

\textit{(iii)} Assume that $\overline{y}\in B\left(y,2^{-1}r_{1}\right)$.
Since $\sum^{p-1}_{n=1}a_{n}<2^{-1}r_{1}$, we obtain

\begin{equation}
\left(\overline{y}-\sum^{p-1}_{n=1}\alpha_{n}\left\Vert v^{n}\right\Vert ^{-1}v^{n}\right)\in B\left(y,r_{1}\right).\label{eq:-27}
\end{equation}
It is true that
\begin{equation}
v^{p}\in\partial\phi\left(\overline{y}-\sum^{p-1}_{n=1}\alpha_{n}\left\Vert v^{n}\right\Vert ^{-1}v^{n}\right).\label{eq:-28}
\end{equation}
Let $z\in V$. The proof of (\ref{eq:-25}) is by induction on $p$.
Clearly, (\ref{eq:-25}) is true for $p=0$. Assume that $p>0$. Then
by the induction hypothesis, (\ref{eq:-27}), (\ref{eq:-28}) and
\textit{(ii)} above,
\begin{align*}
\left\Vert \overline{y}-\sum^{p}_{n=1}\alpha_{n}\left\Vert v^{n}\right\Vert ^{-1}v^{n}-z\right\Vert ^{2} & =\left\Vert \overline{y}-\sum^{p-1}_{n=1}\alpha_{n}\left\Vert v^{n}\right\Vert ^{-1}v^{n}-\alpha_{p}\left\Vert v^{p}\right\Vert ^{-1}v^{p}-z\right\Vert ^{2}\\
 & =\left\Vert \overline{y}-\sum^{p-1}_{n=1}\alpha_{n}\left\Vert v^{n}\right\Vert ^{-1}v^{n}-z\right\Vert ^{2}\\
 & +2\alpha_{p}\left\langle \left\Vert v^{p}\right\Vert ^{-1}v^{p},z-\left(\overline{y}-\sum^{p-1}_{n=1}\alpha_{n}\left\Vert v^{n}\right\Vert ^{-1}v^{n}\right)\right\rangle +\alpha^{2}_{p}\\
 & \le\left\Vert \overline{y}-z\right\Vert ^{2}-\sum^{p-1}_{n=1}\left(r_{2}-\alpha_{n}\right)\alpha_{n}-\alpha_{p}r_{2}+\alpha^{2}_{p}\\
 & =\left\Vert \overline{y}-z\right\Vert ^{2}-\sum^{p}_{n=1}\left(r_{2}-\alpha_{n}\right)\alpha_{n}.
\end{align*}
Lemma \ref{aux_lem} is now proved.
\end{proof}

The following key ``theorem of alternatives'' provides insights
into the behavior of the superiorized algorithm (Algorithm \ref{alg:GDSA}
above) in comparison with the feasibility-seeking one (Algorithm \ref{supGDSA}
above). This result is a new theorem and differs fundamentally from
the theorem of alternatives in \cite[Theorem 4.1]{CZ_sup} which establishes
Fej{\'e}r-monotonicity of the sequence generated by the superiorized
algorithm in the consistent case, under the assumption that the algorithm
does not outperform the unperturbed one and that only metric projections
are employed.

In contrast, Theorem \ref{alg_compar} below compares, under the above
assumption, the distances between the iterates of the sequences produced
by the superiorized algorithm with those of the unperturbed algorithm
in the inconsistent case, for more general operators $\left\{ T_{\left(\varOmega_{k},\omega_{k}\right)}\right\} ^{\infty}_{k=0}$
defined in the previous section. This comparison enables us to derive
sufficient conditions for addressing the guarantee question of the
superiorization method in the present setting.
\begin{thm}
\label{alg_compar}Suppose that $\phi:\mathcal{H}\rightarrow\mathbb{R}$
is a convex and continuous function and $\left\{ \lambda_{k}\right\} ^{\infty}_{k=0}\subset\left[\varepsilon,1+\rho_{\left\{ U_{i}\right\} ^{m}_{i=1}}-\varepsilon\right]$,
where $\varepsilon>0$, is a sequence of positive numbers. Let $x^{0}\in\mathcal{H}$
and $\left\{ x^{k}\right\} ^{\infty}_{k=0}$ be a sequence generated
by Algorithm \ref{alg:GDSA} with respect to the sequence $\left\{ \lambda_{k}\right\} ^{\infty}_{k=0}$,
converging strongly to a point $x\in C$. Assume that $y^{0}=x^{0}$
and that the sequence $\left\{ y^{k}\right\} ^{\infty}_{k=0}$, generated
by the corresponding Algorithm \ref{supGDSA}, converges strongly
to a point $y\in C$. Then exactly one of the following two alternatives
holds:
\begin{enumerate}
\item $\phi\left(y\right)\le\phi\left(x\right)$.

\suspend{enumerate} ~~or \resume{enumerate}
\item $\phi\left(y\right)>\phi\left(x\right)$ and there exists $k_{0}\in\mathbb{N}$
such that $\left\Vert y^{k+1}-x^{k+1}\right\Vert <\left\Vert y^{k}-x^{k}\right\Vert $
for all natural $k\ge k_{0}$.
\end{enumerate}
\end{thm}

\begin{proof}
Assume that $\left\{ y^{k}\right\} ^{\infty}_{k=0}$ converges strongly
to a point $y$ such that $\phi\left(y\right)>\phi\left(x\right)$.
By Lemma \ref{aux_lem}, there exist real numbers $r_{1}>0$, $r_{2}>0$
and a neighborhood $V$ of $x$ such that each $\overline{y}\in B\left(y,r_{1}\right)$
and $v\in\partial\phi\left(\overline{y}\right)$ satisfy its assertions.
By using the strong convergence of $\left\{ y^{k}\right\} ^{\infty}_{k=0}$
to $y$, the strong convergence of $\left\{ x^{k}\right\} ^{\infty}_{k=0}$
to $x$ and the convergence of the series $\sum^{\infty}_{k=0}\sum^{N_{k}}_{n=1}\beta_{k,n}$,
choose $k_{0}\in\mathbb{N}$ such that 
\begin{equation}
y^{k}\in B\left(y,2^{-1}r_{1}\right)\,\,\mathrm{and}\,\,x^{k}\in V\label{eq:-30}
\end{equation}
 and 
\begin{equation}
\sum^{N_{k}}_{n=1}\beta_{k,n}<\min\left\{ 2^{-1}r_{1},r_{2}\right\} \label{eq:-14}
\end{equation}
for each integer $k\ge k_{0}$. This yields, for each $k\ge k_{0}$,
\[
y^{k}+\sum^{n-1}_{i=1}\beta_{k,i}v^{k,i}\in B\left(y,r_{1}\right)
\]
for each $n=1,2,\dots,N_{k}$, and consequently, by Lemma \ref{aux_lem}\textit{(i)},
\begin{equation}
0\not\in\partial\phi\left(y^{k}+\sum^{n-1}_{i=1}\beta_{k,i}v^{k,i}\right)\label{eq:-17}
\end{equation}
for each $n=1,2,\dots,N_{k}$. Let $k\ge k_{0}$ be an integer. By
(\ref{eq:-6}) and (\ref{eq:-17}), 
\begin{equation}
v^{k,n}=-\left\Vert s^{k,n-1}\right\Vert ^{-1}s^{k,n-1},\label{eq:-29}
\end{equation}
where 
\begin{equation}
s^{k,n-1}\in\partial\phi\left(y^{k}+\sum^{n-1}_{i=1}\beta_{k,i}v^{k,i}\right)\label{eq:-26}
\end{equation}
for each $n=1,2,\dots,N_{k}$. Set $p:=N_{k}$, $z:=x^{k}$ and $\overline{y}:=y^{k}$.
For each $n=1,2,\dots,p$, define $\alpha_{n}:=\beta_{k,n}>0$ and
$v^{n}:=s^{k,n-1}$. Then, by (\ref{eq:-30}), (\ref{eq:-14}), (\ref{eq:-17}),
(\ref{eq:-29}) and (\ref{eq:-26}), we have that $\overline{y}\in B\left(y,2^{-1}r_{1}\right)$,
$z\in V$, $\sum^{p}_{n=1}\alpha_{n}<2^{-1}r_{1}$, $\left\{ v^{n}\right\} ^{p}_{n=1}\subset\mathcal{H}\backslash\left\{ 0\right\} $
and $v^{n}\in\partial\phi\left(\overline{y}-\sum^{n-1}_{i=1}\alpha_{i}\left\Vert v^{i}\right\Vert ^{-1}v^{i}\right)$
for each $n=1,2,\dots,p$. Since the operator $T_{\left(\varOmega_{k},\omega_{k}\right)\lambda_{k}}$,
which is the $\lambda_{k}$-relaxation of $T_{\left(\varOmega_{k},\omega_{k}\right)}$,
is nonexpansive by Lemma \ref{lem:relax_output} and $x^{k+1}=T_{\left(\varOmega_{k},\omega_{k}\right)\lambda_{k}}\left(z\right)$,
we obtain from Lemma \ref{aux_lem}\textit{(iii)} that

\begin{align*}
\left\Vert y^{k+1}-x^{k+1}\right\Vert ^{2} & =\left\Vert T_{\left(\varOmega_{k},\omega_{k}\right)\lambda_{k}}\left(y^{k}+\sum^{N_{k}}_{n=1}\beta_{k,n}v^{k,n}\right)-T_{\left(\varOmega_{k},\omega_{k}\right)\lambda_{k}}\left(z\right)\right\Vert ^{2}\le\left\Vert y^{k}+\sum^{N_{k}}_{n=1}\beta_{k,n}v^{k,n}-z\right\Vert ^{2}\\
 & =\left\Vert \overline{y}-\sum^{p}_{n=1}\alpha_{n}\left\Vert v_{n}\right\Vert ^{-1}v_{n}-z\right\Vert ^{2}\le\left\Vert \overline{y}-z\right\Vert ^{2}-\sum^{p}_{n=1}\left(r_{2}-\alpha_{n}\right)\alpha_{n}\\
 & =\left\Vert y^{k}-x^{k}\right\Vert ^{2}-\sum^{N_{k}}_{n=1}\left(r_{2}-\beta_{k,n}\right)\beta_{k,n}.
\end{align*}
Since $\sum^{N_{k}}_{n=1}\beta_{k,n}<r_{2}$ by (\ref{eq:-14}), the
result follows and the proof of the theorem is complete.
\end{proof}

Based on negation of the alternative (ii) of this new theorem of alternatives,
we derive sufficient conditions in the next corollary.
\begin{cor}
\label{subseq_suff_cond}Let $\left\{ \lambda_{k}\right\} ^{\infty}_{k=0}\subset\left[\varepsilon,1+\rho_{\left\{ U_{i}\right\} ^{m}_{i=1}}-\varepsilon\right]$,
where $\varepsilon>0$, be a sequence of positive numbers. Let $x^{0}\in\mathcal{H}$
and $\left\{ x^{k}\right\} ^{\infty}_{k=0}$ be a sequence generated
by Algorithm \ref{alg:GDSA} with respect to the sequence $\left\{ \lambda_{k}\right\} ^{\infty}_{k=0}$
converging strongly to a point $x\in C$. Assume that $y^{0}=x^{0}$
and that the sequence $\left\{ y^{k}\right\} ^{\infty}_{k=0}$, generated
by the corresponding Algorithm \ref{supGDSA} converges strongly to
a point $y\in C$. Suppose that there exists a strictly increasing
sequence $\left\{ \gamma_{k}\right\} ^{\infty}_{k=0}$ of natural
numbers such that 
\begin{equation}
\left\Vert T_{\left(\varOmega_{\gamma_{k}},\omega_{\gamma_{k}}\right)\lambda_{\gamma_{k}}}\left(y^{\gamma_{k}}+\sum^{N_{\gamma_{k}}}_{n=1}\beta_{\gamma_{k},n}v^{\gamma_{k},n}\right)-T_{\left(\varOmega_{\gamma_{k}},\omega_{\gamma_{k}}\right)\lambda_{\gamma_{k}}}\left(x^{\gamma_{k}}\right)\right\Vert \ge\left\Vert y^{\gamma_{k}}-x^{\gamma_{k}}\right\Vert \label{eq:-3}
\end{equation}
for each $k\in\mathbb{N}$. Then $\phi\left(y\right)\le\phi\left(x\right)$.
\end{cor}

\begin{proof}
Clearly,
\[
y^{\gamma_{k}+1}=T_{\left(\varOmega_{\gamma_{k}},\omega_{\gamma_{k}}\right)\lambda_{\gamma_{k}}}\left(y^{\gamma_{k}}+\sum^{N_{\gamma_{k}}}_{n=1}\beta_{\gamma_{k},n}v^{\gamma_{k},n}\right).
\]
The result now follows from Theorem \ref{alg_compar}.
\end{proof}

The following corollary provides a stronger sufficient condition for
the superiorization guarantee question.
\begin{cor}
\label{inner_prod}Let $\phi:\mathcal{H}\rightarrow\mathbb{R}$ be
a convex and continuous function and $\left\{ \lambda_{k}\right\} ^{\infty}_{k=0}\subset\left[\varepsilon,1+\rho_{\left\{ U_{i}\right\} ^{m}_{i=1}}-\varepsilon\right]$,
where $\varepsilon>0$, be a sequence of positive numbers. Assume
that $x^{0}\in\mathcal{H}$ and $\left\{ x^{k}\right\} ^{\infty}_{k=0}$
is a sequence generated by Algorithm \ref{alg:GDSA} with respect
to the sequence $\left\{ \lambda_{k}\right\} ^{\infty}_{k=0}$ converging
strongly to a point $x\in C$. Assume that that the sequence $\left\{ y^{k}\right\} ^{\infty}_{k=0}$,
generated by the corresponding Algorithm \ref{supGDSA} converges
strongly to a point $y\in C.$ Suppose that there exists a strictly
increasing sequence $\left\{ \gamma_{k}\right\} ^{\infty}_{k=0}$
of natural numbers so that for each $k\in\mathbb{N}$, the parameters
$\left\{ \left\{ \beta_{k,n}\right\} ^{N_{k}}_{n=1}\right\} ^{\infty}_{k=0}$
are such that 
\begin{equation}
\begin{alignedat}{1}\begin{split}\left\langle y^{\gamma_{k}}-T_{\left(\varOmega_{\gamma_{k}},\omega_{\gamma_{k}}\right)\lambda_{\gamma_{k}}}\left(x^{\gamma_{k}}\right),y^{\gamma_{k}}-T_{\left(\varOmega_{\gamma_{k}},\omega_{\gamma_{k}}\right)\lambda_{\gamma_{k}}}\left(y^{\gamma_{k}}+\sum^{N_{\gamma_{k}}}_{n=1}\beta_{\gamma_{k},n}v^{\gamma_{k},n}\right)\right\rangle \\
\le\left\langle y^{\gamma_{k}}-x^{\gamma_{k}},x^{\gamma_{k}}-T_{\left(\varOmega_{\gamma_{k}},\omega_{\gamma_{k}}\right)\lambda_{\gamma_{k}}}\left(x^{\gamma_{k}}\right)\right\rangle .
\end{split}
\end{alignedat}
\label{eq:}
\end{equation}
Then $\phi\left(y\right)\le\phi\left(x\right)$.
\end{cor}

\begin{proof}
By (\ref{eq:}), we see that 
\begin{equation}
\left\langle y^{\gamma_{k}}-x^{\gamma_{k}+1},y^{\gamma_{k}}-y^{\gamma_{k}+1}\right\rangle \le\left\langle y^{\gamma_{k}}-x^{\gamma_{k}},x^{\gamma_{k}}-x^{\gamma_{k}+1}\right\rangle \label{eq:-1}
\end{equation}
Now we use the properties of the inner product and (\ref{eq:-1})
to deduce 
\begin{align*}
\left\Vert y^{\gamma_{k}+1}-x^{\gamma_{k}+1}\right\Vert ^{2} & =\left\Vert y^{\gamma_{k}+1}-y^{\gamma_{k}}\right\Vert ^{2}+\left\Vert y^{\gamma_{k}}-x^{\gamma_{k}+1}\right\Vert ^{2}-2\left\langle y^{\gamma_{k}+1}-y^{\gamma_{k}},x^{\gamma_{k}+1}-y^{\gamma_{k}}\right\rangle \\
 & =\left\Vert y^{\gamma_{k}+1}-y^{\gamma_{k}}\right\Vert ^{2}+\left\Vert y^{\gamma_{k}}-x^{\gamma_{k}}\right\Vert ^{2}+\left\Vert x^{\gamma_{k}}-x^{\gamma_{k}+1}\right\Vert ^{2}\\
 & +2\left\langle y^{\gamma_{k}}-x^{\gamma_{k}},x^{\gamma_{k}}-x^{\gamma_{k}+1}\right\rangle -2\left\langle y^{\gamma_{k}}-x^{\gamma_{k}+1},y^{\gamma_{k}}-y^{\gamma_{k}+1}\right\rangle \\
 & \ge\left\Vert y^{\gamma_{k}}-x^{\gamma_{k}}\right\Vert ^{2}.
\end{align*}
 The result now follows from Corollary \ref{subseq_suff_cond}.
\end{proof}

Although the inequalities in Corollaries \ref{subseq_suff_cond} and
\ref{inner_prod} cannot, in general, be verified during the execution
of the algorithm, their structural meaning suggests practical indicators
that can be monitored or encouraged in implementations.

For example, one may require that each perturbation step increases
the distance between the iterations of the superiorized algorithm
and the feasibility-seeking one, or that each perturbation increases
the angular deviation between the update direction of the superiorized
algorithm and the update direction that the feasibility-seeking algorithm
would produce when starting from the same iterate. In addition, monitoring
feasibility measures along the superiorized iterates compared to the
unperturbed algorithm can serve as an informal check that the perturbations
guide the iterates toward a limit point with an objective function
value that is not larger than that of the limit point of the iterates
of the unperturbed algorithm.

These considerations are reflected in the parameter selection strategies
described in the sequel, which are designed to promote the qualitative
behavior prescribed by the theoretical conditions, even though the
conditions themselves are not explicitly verified.

The following examples illustrate concrete settings in which the abstract
assumptions in Corollaries \ref{subseq_suff_cond} and \ref{inner_prod}
can be checked explicitly.
\begin{example}
[Simultaneous method]\label{SimPro}Let $\varepsilon>0$ be a real
number. For all $k\in\mathbb{N}$, set $q_{k}:=1$, $\varOmega_{k}:=\left\{ 1,2,\dots,m\right\} {}^{\left\{ 1,2,\dots,q_{k}\right\} }$
and $\omega_{k}:=\overline{\omega}$ for a fixed $\overline{\omega}:\left\{ 1,2,\dots,m\right\} {}^{\left\{ 1,2,\dots,q_{k}\right\} }\rightarrow\left(0,1\right]$.
Clearly, for each $i\in\left\{ 1,2,\dots,m\right\} $, there is a
unique string $t^{i}\in\left\{ 1,2,\dots,m\right\} {}^{\left\{ 1,2,\dots,q_{k}\right\} }$
such that $t^{i}\left(1\right)=i$. Hence we can define the mapping
$\omega:\left\{ 1,2,\dots,m\right\} \rightarrow\left(0,1\right]$
by $\omega_{i}:=\overline{\omega}\left(t^{i}\right)$ for each $i\in\left\{ 1,2,\dots,m\right\} $.
In this case Algorithm \ref{alg:GDSA} with the above provisions provides
a fully-simultaneous method, that is, 
\begin{equation}
T_{\left(\varOmega_{k},\omega_{k}\right)}=\sum_{t\in\varOmega_{k}}\overline{\omega}\left(t\right)U_{t\left(1\right)}=\sum^{m}_{i=1}\overline{\omega}\left(t^{i}\right)U_{t^{i}\left(1\right)}=\sum^{m}_{i=1}\omega_{i}U_{i}\label{eq:-34}
\end{equation}
for each $k\in\mathbb{N}$ and $C=F=\mathrm{Fix}\sum^{m}_{i=1}\omega_{i}U_{i}$.
We see from (\ref{eq:-34}) that the sequence $\left\{ T_{\left(\varOmega_{k},\omega_{k}\right)}\right\} ^{\infty}_{k=0}$
is $\limsup$-admissible. Hence, under the assumption that $\mathrm{Fix}\sum^{m}_{i=1}\omega_{i}U_{i}\not=\emptyset$
we obtain, by Theorem \ref{thm_conv}, the weak convergence of this
fully-simultaneous method, with parameters $\left\{ \lambda_{k}\right\} ^{\infty}_{k=0}\subset\left[\varepsilon,1+\rho_{\left\{ U_{i}\right\} ^{m}_{i=1}}-\varepsilon\right]$,
to a point in $C$. In particular, when $U_{i}=P_{C_{i}}$, where
$C_{i}$ is a nonempty, closed and convex subset of $\mathcal{H}$
for each $i=1,2,\dots,m$, we obtain the well-known simultaneous projection
method, see, for example, \cite[Subsection 5.4]{C_book}. In this
case $\rho_{\left\{ U_{i}\right\} ^{m}_{i=1}}=1$ (by Example \ref{metric projection}
and Remark \ref{4.10}(a)), $I=\left\{ \sum^{m}_{i=1}\omega_{i}P_{C_{i}}\right\} $
and 
\begin{equation}
C=F=\mathrm{Fix}\sum^{m}_{i=1}\omega_{i}P_{C_{i}}=\underset{x\in\mathcal{H}}{\mathrm{Argmin}}f\left(x\right),\label{eq:-5}
\end{equation}
where $f:\mathcal{H}\rightarrow\mathbb{R}$ is a, so called, proximity
function defined by $f:=2^{-1}\sum^{m}_{i=1}\omega_{i}\left\Vert P_{C_{i}}-Id\right\Vert ^{2}$.
For the proof of the last equality in (\ref{eq:-5}), see, for instance,
Theorem 4.4.6 in \cite{C_book}. By Theorem \ref{thm_conv}\textit{(i)},
the simultaneous projection method with parameters $\left\{ \lambda_{k}\right\} ^{\infty}_{k=0}\subset\left[\varepsilon,2-\varepsilon\right]$
converges weakly to a point in $\underset{x\in\mathcal{H}}{\mathrm{Argmin}}f\left(x\right)$.
If, in addition, the operator $\sum^{m}_{i=1}\omega_{i}P_{C_{i}}$
is approximately shrinking, then (since the family $\left\{ C_{T}\right\} _{T\in I}=\left\{ C\right\} $
is boundedly regular), this convergence is strong by Theorem \ref{thm_conv}\textit{(ii)}.

Set $\mathcal{H}:=\mathbb{R}^{2}$ with the usual inner product and
$m:=2$. Let $C_{1}:=\left\{ u\in\mathbb{R}^{2}\,|\,u_{1}=1\right\} $
and $C_{2}:=\left\{ u\in\mathbb{R}^{2}\,|\,u_{1}=-1\right\} $ be
two parallel hyperplanes. Define $U_{i}:=P_{C_{i}}$ for each $i=1,2$,
and $\lambda_{k}:=1$ for each $k\in\mathbb{N}$. Assume that $x^{0}:=y^{0}:=\begin{pmatrix}0\\
0
\end{pmatrix}$ and let $\phi:\mathcal{H}\rightarrow\mathbb{R}$ be a convex function
defined by $\phi\left(u\right):=-u_{2}$ for each $u\in\mathcal{H}$.
Clearly, $C_{1}\cap C_{2}=\emptyset$ and the subgradient set of $\phi$
at each point is $\left\{ \begin{pmatrix}0\\
-1
\end{pmatrix}\right\} $. Let $\left\{ \left\{ \beta_{k,n}\right\} ^{N_{k}}_{n=1}\right\} ^{\infty}_{k=0}$
be an arbitrary sequence of positive numbers such that $\sum^{\infty}_{k=0}\sum^{N_{k}}_{n=1}\beta_{k,n}<\infty$.
Set $\overline{\omega}:\left\{ 1,2\right\} {}^{\left\{ 1\right\} }\rightarrow\left(0,1\right]$
by $\overline{\omega}\left(t\right):=2^{-1}$ for each string $t\in\left\{ 1,2\right\} {}^{\left\{ 1\right\} }$
and we obtain the function $\omega:\left\{ 1,2\right\} \rightarrow\left(0,1\right]$,
induced by $\overline{\omega}$, which attains a value of $2^{-1}$
for each $i\in\left\{ 1,2\right\} $. Then 
\[
T_{\left(\varOmega_{k},\omega_{k}\right)\lambda_{k}}=T_{\left(\varOmega_{k},\omega_{k}\right)}=\sum^{m}_{i=1}\omega_{i}P_{C_{i}}=2^{-1}P_{C_{1}}+2^{-1}P_{C_{2}}
\]
and
\[
C=F=\mathrm{Fix}\sum^{2}_{i=1}2^{-1}P_{C_{i}}=\underset{x\in\mathcal{H}}{\mathrm{Argmin}}f\left(x\right).
\]
In these settings the sequence $\left\{ x^{k}\right\} ^{\infty}_{k=0}$
generated by Algorithm \ref{alg:GDSA} satisfies $x^{k}=\begin{pmatrix}0\\
0
\end{pmatrix}$ for each $k\in\mathbb{N}$ and the sequence $\left\{ y^{k}\right\} ^{\infty}_{k=0}$
generated by Algorithm \ref{supGDSA} satisfies $y^{k}=\begin{pmatrix}0\\
\sum^{k-1}_{i=0}\sum^{N_{k-1}}_{n=1}\beta_{i,n}
\end{pmatrix}$ for each $k\in\mathbb{N}$ (recalling that by definition, $\sum^{k-1}_{i=0}\sum^{N_{k-1}}_{n=1}\beta_{i,n}=0$
for $k=0$). As a result

\begin{align*}
\left\Vert \sum^{2}_{i=1}2^{-1}P_{C_{i}}\left(y^{k}+\sum^{N_{k}}_{n=1}\beta_{k,n}v^{k,n}\right)-\sum^{2}_{i=1}2^{-1}P_{C_{i}}\left(x^{k}\right)\right\Vert  & =\left\Vert y^{k+1}-x^{k+1}\right\Vert =\sum^{k}_{i=0}\sum^{N_{k}}_{n=1}\beta_{i,n}\\
 & >\sum^{k-1}_{i=0}\sum^{N_{k-1}}_{n=1}\beta_{i,n}=\left\Vert y^{k}-x^{k}\right\Vert 
\end{align*}
and
\begin{align*}
\left\langle y^{k}-T_{\left(\varOmega_{k},\omega_{k}\right)\lambda_{k}}\left(x^{k}\right),y^{k}-T_{\left(\varOmega_{k},\omega_{k}\right)\lambda_{k}}\left(y^{k}+\sum^{N_{k}}_{n=1}\beta_{k,n}v^{k,n}\right)\right\rangle \\
=\left\langle y^{k}-x^{k+1},y^{k}-y^{k+1}\right\rangle \\
=\left\langle \begin{pmatrix}0\\
\sum^{k-1}_{i=0}\sum^{N_{k-1}}_{n=1}\beta_{i,n}
\end{pmatrix},\begin{pmatrix}0\\
-\sum^{N_{k}}_{n=1}\beta_{k,n}
\end{pmatrix}\right\rangle \\
-\sum^{k-1}_{i=0}\sum^{N_{k-1}}_{n=1}\beta_{i,n}\cdot\sum^{N_{k}}_{n=1}\beta_{k,n}<0\\
=\left\langle y^{k}-x^{k},x^{k}-T_{\left(\varOmega_{k},\omega_{k}\right)\lambda_{k}}\left(x^{k}\right)\right\rangle  & .
\end{align*}
for each $k\in\mathbb{N}$. By setting $\gamma_{k}:=k$ for each $k\in\mathbb{N}$,
we see that either of the inequalities which appear in Corollary \ref{subseq_suff_cond}
and Corollary \ref{inner_prod} are satisfied. Therefore, the superiorization
algorithm performs at least as well as the unperturbed one. Indeed,
denote $L:=\sum^{\infty}_{k=0}\sum^{N_{k}}_{n=1}\beta_{k,n}>0$. Then
$\lim_{k\rightarrow\infty}x^{k}=\begin{pmatrix}0\\
0
\end{pmatrix}$ and $\lim_{k\rightarrow\infty}y^{k}=\begin{pmatrix}0\\
L
\end{pmatrix}$, while 
\[
\phi\begin{pmatrix}0\\
L
\end{pmatrix}=-L<0=\phi\begin{pmatrix}0\\
0
\end{pmatrix}.
\]

More generally, in the above settings assume that $\phi$ is a convex
and continuously differentiable function such that $\frac{\partial\phi}{\partial u_{2}}\left(y^{0}\right)\not=0$.
By using the continuous differentiability of $\phi$, choose the perturbation
steps $\left\{ \beta_{k,n}\right\} ^{N_{k}}_{n=1}$ small enough so
that $\sum^{N_{k}}_{n=1}\beta_{k,n}<k^{-2}$ and 
\begin{equation}
\frac{\partial\phi}{\partial u_{2}}\left(y^{k}\right)\cdot\frac{\partial\phi}{\partial u_{2}}\left(\begin{pmatrix}0\\
y^{k}_{2}+\sum^{n}_{i=1}\beta_{k,i}v^{k,i}_{2}
\end{pmatrix}\right)>0\label{eq:-2}
\end{equation}
 for each $n=1,2,\dots,N_{k}$, where $v^{k,n}$ is defined recursively
by (\ref{eq:-6}), that is, 
\[
v^{k,i}:=-\nabla\phi\left(y^{k}+\sum^{i-1}_{j=1}\beta_{k,i}v^{k,j}\right)\left\Vert \nabla\phi\left(y^{k}+\sum^{i-1}_{j=1}\beta_{k,i}v^{k,j}\right)\right\Vert ^{-1}
\]
for each $i=1,2,\dot{\dots,n}$.

Since for each $k\in\mathbb{N}$, $y^{k}_{1}=0$, we see, in the light
of (\ref{eq:-2}), that for all $k\in\mathbb{N}$, either $y^{k}_{2}\ge0$
(if $\frac{\partial\phi}{\partial u_{2}}\left(y^{0}\right)<0$) or
$y^{k}_{2}\le0$ (if $\frac{\partial\phi}{\partial u_{2}}\left(y^{0}\right)>0$)
and

\[
\left\Vert y^{k+1}-x^{k+1}\right\Vert =\left\Vert y^{k+1}\right\Vert =\left|y^{k}_{2}+\sum^{N_{k}}_{i=1}\beta_{k,i}v^{k,i}_{2}\right|>\left|y^{k}_{2}\right|=\left\Vert y^{k}\right\Vert =\left\Vert y^{k}-x^{k}\right\Vert .
\]
and 
\[
\left\langle y^{k}-T_{\left(\varOmega_{k},\omega_{k}\right)\lambda_{k}}\left(x^{k}\right),y^{k}-T_{\left(\varOmega_{k},\omega_{k}\right)\lambda_{k}}\left(y^{k}+\sum^{N_{k}}_{n=1}\beta_{k,n}v^{k,n}\right)\right\rangle 
\]
\begin{align*}
=\left\langle y^{k}-x^{k+1},y^{k}-y^{k+1}\right\rangle =\left\langle \begin{pmatrix}0\\
y^{k}_{2}
\end{pmatrix},\begin{pmatrix}0\\
-\sum^{N_{k}}_{i=1}\beta_{k,i}v^{k,i}_{2}
\end{pmatrix}\right\rangle \le0\\
=\left\langle y^{k}-x^{k},x^{k}-T_{\left(\varOmega_{k},\omega_{k}\right)\lambda_{k}}\left(x^{k}\right)\right\rangle .
\end{align*}
Again by setting $\gamma_{k}:=k$ for each $k\in\mathbb{N}$, we see
that either of the inequalities which appear in Corollary \ref{subseq_suff_cond}
and Corollary \ref{inner_prod} are satisfied. Therefore, the superiorization
algorithm performs no worse than the unperturbed one. Observe that
the above discussion extends to any finite-dimensional Hilbert space
$\mathcal{H}$.
\end{example}

We next discuss how practitioners should test the step-sizes of the
perturbations in the SM to follow the sufficient conditions presented
here in quest of the superiorization guarantee problem. While Conditions
(\ref{eq:-3}) and (\ref{eq:}) provide a rigorous guarantee that
the superiorized algorithm performs at least as well as the unperturbed
algorithm, they are formulated in terms that are not always directly
computable in practice. Nevertheless, these conditions admit a clear
structural interpretation: they require that each perturbation step
remains compatible with the underlying convergence mechanism of the
feasibility-seeking algorithm and not compromise convergence with
respect to the objective function value at the limit point.

Motivated by this interpretation, we propose practical strategies
for selecting the perturbations that aim to emulate the behavior prescribed
by (\ref{eq:-3}) and (\ref{eq:}). Although these strategies do not
verify the assumptions explicitly, they are designed so that, on an
iteration-by-iteration basis, the superiorized algorithm follows the
corresponding theoretical model as closely as possible.

The role of the theoretical conditions (\ref{eq:-3}) and (\ref{eq:})
is, therefore, not merely to certify correctness in special cases,
but rather to provide insight into how perturbations should be chosen
so as to remain aligned with the convergence behavior of the unperturbed
algorithm while promoting improved objective function values. The
practical strategies proposed here are steered by this insight.

Assume that $\left\{ \delta_{k}\right\} ^{\infty}_{k=0}$ and $\left\{ \tau_{k}\right\} ^{\infty}_{k=0}$
are sequences of positive numbers such that $\sum^{\infty}_{k=1}\delta_{k}<\infty$
and $\tau_{k}<\delta_{k}$ for each $k\in\mathbb{N}$. Choose $x^{0}:=y^{0}$
as arbitrary starting points. For each $k\in\mathbb{N}$ such that
$k\,\mathrm{mod}\,2=0$, let $\beta_{k}$ be a selection of the set
\[
\underset{\beta\in\left[\tau_{k},\delta_{k}\right]^{N_{k}}\,|\,\sum^{N_{k}}_{n=1}\beta_{n}\le\delta_{k}}{\mathrm{Argmin}}\left\Vert T_{\left(\varOmega_{k},\omega_{k}\right)\lambda_{k}}\left(y^{k}+\sum^{N_{k}}_{n=1}\beta^{n}v^{k,n}\right)-T_{\left(\varOmega_{k},\omega_{k}\right)\lambda_{k}}\left(x^{k}\right)\right\Vert ,
\]
and for each $k\in\mathbb{N}$ such that $k\,\mathrm{mod}\,2=1$,
let $\beta_{k}$ be a selection of the set
\[
\underset{\beta\in\left[\tau_{k},\delta_{k}\right]^{N_{k}}\,|\,\sum^{N_{k}}_{n=1}\beta_{n}\le\delta_{k}}{\mathrm{Argmax}}\left\Vert T_{\left(\varOmega_{k},\omega_{k}\right)\lambda_{k}}\left(y^{k}+\sum^{N_{k}}_{n=1}\beta^{n}v^{k,n}\right)-T_{\left(\varOmega_{k},\omega_{k}\right)\lambda_{k}}\left(x^{k}\right)\right\Vert .
\]
Let $\left\{ x^{k}\right\} ^{\infty}_{k=0}$ be a sequence generated
by Algorithm \ref{alg:GDSA} and let $\left\{ y^{k}\right\} ^{\infty}_{k=0}$
be a sequence generated by Algorithm \ref{supGDSA}, where $\beta_{k,n}:=\beta^{n}_{k}$
for each $k\in\mathbb{N}$ and each $n=1,\dots,N_{k}$. In this way,
choosing minimal possible values of $\left\Vert y^{k}-x^{k}\right\Vert $
in odd iterations and maximal possible values of $\left\Vert y^{k}-x^{k}\right\Vert $
in even ones, we follow the theoretical model described in Corollary
\ref{subseq_suff_cond}.

Alternatively, in the above settings for each $k\in\mathbb{N}$, we
can choose $\beta_{k}$ to be a selection of the set
\[
\underset{\beta\in\left[\tau_{k},\delta_{k}\right]^{N_{k}}\,|\,\sum^{N_{k}}_{n=1}\beta_{n}\le\delta_{k}}{\mathrm{Argmin}}\left\langle y^{k}-T_{\left(\varOmega_{k},\omega_{k}\right)\lambda_{k}}\left(x^{k}\right),y^{k}-T_{\left(\varOmega_{k},\omega_{k}\right)\lambda_{k}}\left(y^{k}+\sum^{N_{k}}_{n=1}\beta^{n}v^{k,n}\right)\right\rangle 
\]
and generate sequences $\left\{ x^{k}\right\} ^{\infty}_{k=0}$ and
$\left\{ y^{k}\right\} ^{\infty}_{k=0}$ by, respectively, Algorithm
\ref{alg:GDSA} and Algorithm \ref{supGDSA}, where in Algorithm \ref{supGDSA}
$\beta_{k,n}:=\beta^{n}_{k}$ for each $k\in\mathbb{N}$ and each
$n=0,1,\dots,N_{k}$. In this way we minimize the left side of (\ref{eq:})
and thus follow the theoretical model described in Corollary \ref{inner_prod}. 

\selectlanguage{english}%
\bigskip{}
\foreignlanguage{american}{}
\selectlanguage{american}%
\textbf{Acknowledgments}. We thank Walaa Moursi and Henry Wolkowicz
from the Department of Combinatorics and Optimization at the University
of Waterloo for some fruitful inspiring discussions in the early stages
of this investigation. We gratefully acknowledge the enlightening
and constructive comments of the two anonymous referees which helped
us improve the paper.

\selectlanguage{english}%
\bigskip{}

\textbf{Author Contribution declaration.} Both authors Kay Barshad
and Yair Censor worked on the research, on developing the ideas and
both wrote the manuscript and reviewed it.

\bigskip{}

\textbf{Funding. }The work of Kay Barshad and Yair Censor is supported
by the Cooperation Program in Cancer Research of the German Cancer
Research Center (DKFZ) and Israel's Ministry of Innovation, Science
and Technology (MOST). The work of Yair Censor was supported also
by the U.S. National Institutes of Health Grant Number R01CA266467.

\bigskip{}

\textbf{Data Availability.} \foreignlanguage{american}{No data was
used for the research described in the article.}\bigskip{}

\textbf{Compliance with Ethical Standards. }The authors have no potential
conflicts of interest to declare that are relevant to the content
of this article.

\bigskip{}

\textbf{Competing Interests. }The authors have no competing interests
to declare that are relevant to the content of this article. 

\bigskip{}

\bibliographystyle{j_style}
\bibliography{bank_of_references}
\selectlanguage{american}%

\end{document}